\documentclass[11pt]{amsart}   % For Latex2e
\usepackage{amssymb,xy}
\usepackage{amsthm, amsmath}
\usepackage[latin1]{inputenc}
\usepackage[T1]{fontenc}
\xyoption{all}
\theoremstyle{plain}
\newtheorem{theorem}{Theorem}[section]
\newtheorem{lemma}[theorem]{Lemma}
\newtheorem{proposition}[theorem]{Proposition}
\newtheorem{corollary}[theorem]{Corollary}

\theoremstyle{definition}
\newtheorem{definition}[theorem]{Definition}
\newtheorem{example}[theorem]{Example}
\newtheorem{algorithm}[theorem]{Algorithm}

\theoremstyle{remark}
\newtheorem{remark}[theorem]{Remark}

\newcommand{\scst}{\scriptscriptstyle}

\title{On permanents of  Sylvester Hadamard matrices}
\author{
Jos\'e Andr\'es Armario
}
\date{}

\address{\small \rm  Depto Matemática Aplicada I\\
 Universidad de Sevilla\\ Avda. Reina Mercedes s/n 41012 Sevilla\\ Spain.}
\email{ armario@us.es}

\begin{document}

\begin{abstract}
It is well-known that the evaluation of the permanent of an arbitrary $(-1,1)$-matrix is a formidable problem. Ryser's formula is one of the fastest known general algorithms for computing permanents. In this paper, Ryser's formula has been rewritten for the special case of Sylvester Hadamard matrices by using  their cocyclic construction. The rewritten formula presents an important reduction in the number of sets of $r$ distinct rows of the matrix to be considered. However, the algorithm needs a preprocessing part which remains  time-consuming in general.
\end{abstract}

\maketitle

\section{Introduction}

A {\it Hadamard matrix} $H$ of order $n$ is an $n\times n$ matrix with elements $\pm 1$ and $H H^T=nI$. A Hadamard matrix is said to be normalized if it has its first row and column all $1's$. We can always normalize a Hadamard matrix by multiplying rows and columns by $-1$.
 It is well-known that  $n$   is either $2$ or a  multiple of 4 and
it is conjectured that Hadamard matrices exist for every $n \equiv 0 \mod 4$ (see \cite{Hor07}).

 Sylvester in 1867 noted that given a Hadamard matrix $H$ of order $n$, then $$\left[\begin{array}{cr} H & H \\ H & -H\end{array}\right]$$ is a Hadamard matrix of order $2n$. Matrices of this form are called {\em Sylvester Hadamard} and defined for all powers of 2. Below is given the Sylvester Hadamard matrix of order 2
$$H_2=\left[\begin{array}{cr} 1 &  1\\ 1 & -1\end{array}\right].$$
Starting with $H_2$, Sylvester Hadamard matrices of order $2^k$ can be formed by $\stackrel{k-copies}{H_2\times\cdots\times H_2}$ the Kronecker product of $k$ copies of $H_2$ and are denoted   $H_{2^{k}}$.

Two Hadamard matrices $H$ and $H'$ are called {\em equivalent} (or Hadamard equivalent, or $H$-equivalent) if one can be obtained from the other by a sequence of row and/or column interchanges and row and/or column negations. The  question of classifying  Hadamard matrices of order $n> 32$ remains unanswered and only partial results are known. %\cite{Kou}.

We recall that the original interest in Hadamard matrices stemmed from the fact that a Hadamard matrix $H=[h_{ij}]$ of order $n$ satisfies equality in Hadamard's inequality

$$\left(\mbox{det}\, H\right)^2\leq \prod_{j=1}\sum_{i=1}|h_{ij}|^2$$
for entries in the unit circle.

The permanent is a matrix function introduced (independently) by Cauchy and Binet in 1812.

\begin{definition}
Let $N$ be the set $\{1,\ldots,n\}$, ($n\in {\bf Z}^{+}$). The symmetric group $S_n$ is the group of all $n!$ permutations of $N$.
The {\em permanent} of an $n\times n$ matrix $A=\left[a_{ij}\right]$  is defined by
$$\mbox{per}\,(A)=\sum_{\sigma\in S_n}\prod_{i=1}^{n} a_{i,\sigma(i)}.$$
\end{definition}

At first sight
it seems to be a simplified version of the determinant, but this impression is misleading.
For instance, the determinant of an arbitrary matrix can be evaluated efficiently using Gaussian elimination, however the computation of the permanent is much more complicated. Valiant \cite{Val79} proved that it belongs to the class of $\sharp$P-complete problems, which basically means that there is almost no possibility of finding a polynomial time deterministic algorithm for computing the permanent in general.

Hadamard matrices of order $n$ are characterized by attaining the maximal possible absolute value of the determinant, $n^{n/2}$, among all
square matrices of order $n$ with entries from the unit disc,
but few things are known on their permanents, aside from $| \mbox{per}\,(H)| \leq |\mbox{det}\,(H)| = n^{n/2}$. Moreover, the permanent of a Hadamard matrix has hardly been worked on.
For what we know,  the permanents for all Hadamard matrices of orders up to and including 28 were calculated in \cite{Wan05}, but for orders greater than 28 the permanents remains unknown.

In 1974 Wang \cite{Wan74} posed the following question:

\bigskip
\noindent{\bf Problem 1} $\quad$
 Can the permanent of  a Hadamard matrix $H$ of order $n$ vanish for $n>2$?
\bigskip

This problem, which is listed as problem $\sharp$5 in Minc's catalogue of open problems \cite{Min78} has in any case been answered in the negative for $n<32$ by
Wanless \cite{Wan05}. Moreover,
let $f(n)=\sum_{k\geq 1}\lfloor2^{-k}n\rfloor$ be the greatest exponent $e$ such that $2^e$ divides $n!$. Wanless
 conjectured that $2^{f(n)}$
divides $per (H)$ and  $2^{f (n)+1}$ does not.

The Sylvester matrices have revealed to be the most treatable class of Hadamard matrices.
For instance, the eigenvalues, the eigenvectors and  the growth factor of  Hadamard matrices are unknown in general but they are well-known for Sylvester Hadamard matrices \cite{YH82,DP88}.

In the following result, we provide a sufficient condition to give a negative answer to Problem 1 for Sylvester Hadamard matrices.

\begin{proposition}\label{condicionsuficiente}
If
\begin{equation}\label{demsh}
\mbox{per}\,((H_{2^n}\times I_{2^m})(I_{2^n}\times H_{2^m}))\geq \mbox{per}\,(H_{2^n}\times I_{2^m})\,\mbox{per}\,(I_{2^n}\times H_{2^m})
\end{equation}
for $n\geq 2$ and $m\geq 2$, then the permanent of a Sylvester Hadamard matrix of order $2^{k}$ never vanishes  for $k>2$.
\end{proposition}
\begin{proof}
It is well-known that
$$H_{2^{n+m}}=\stackrel{\tiny {2^{n+m}-copies}}{H_2\times\cdots\times H_2}$$
(using the associative law for the Kronecker product)
$$\quad\,\,\,\,\,=H_{2^n}\times H_{2^m}$$
(taking into account that $A\times B=(A\times I_m)(I_n\times B)$ for $A$ and $B$ squares matrices of order $n$ and $m$, respectively)
$$\qquad\qquad\qquad\qquad=(H_{2^n}\times I_{2^m})(I_{2^n}\times H_{2^m}).$$
Let us recall the following basic properties for permanents \cite{Bru66},
\begin{enumerate}
\item $\mbox{per}\,(I_n\times A)=(\mbox{per}\,(A))^n$.
\item $\mbox{per}\,(A\times B)= \mbox{per}\,(B\times A)$.
\end{enumerate}
Therefore, if the inequality (\ref{demsh}) occurs then
$$ \mbox{per}\, (H_{2^{n+m}})\geq (\mbox{per}\,(H_{2^n}))^{2^m} (\mbox{per}\,(H_{2^m}))^{2^n},\quad \mbox{for}\,\, n\geq 2\,\, \mbox{and}\,\, m\geq 2.$$
Taking into consideration the inequality above and  that $\mbox{per}\,(H_4)=8$ and $\mbox{per}\,(H_8)=384$, it follows the desired result.

\end{proof}

The following open problem arises naturally from Proposition \ref{condicionsuficiente}.

\bigskip
\noindent{\bf Problem 2} $\quad$
 For each $n\geq 2$ and $m\geq 2$ is the inequality (\ref{demsh})  always true?
\bigskip

It is well-known that  it is not in general true that $$\mbox{per}\, (AB)=\mbox{per}\,(A) \,\mbox{per}\,(B)$$ for $A$ and $B$ square matrices. This is a devastating blow, since most of the nice properties of
determinants follow from $\mbox{det}\,(AB) = \mbox{det}\,(A) \mbox{det}\,(B)$. In particular, we cannot use Gaussian elimination to calculate
permanents.
For the special case when the matrices are nonnegative, Brualdi \cite{Bru66} proved that
$$\mbox{per}\, (AB)\geq \mbox{per}\,(A) \,\mbox{per}\,(B).$$ On the other hand, Problem 2 is conceivably much more difficult.

In the early 90s, a surprising link between homological algebra and Hadamard matrices \cite{HD94} led to the study of cocyclic Hadamard matrices. Hadamard matrices of many classes are revealed to be (equivalent to) cocyclic matrices \cite{Hor07}. Among them are Sylvester Hadamard matrices, Williamson-type Hadamard matrices and Paley Hadamard matrices.

Let
 $G=\{g_1=1,\,g_2,\ldots,g_{4t}\}$ be a multiplicative group, not necessarily abelian. Functions
$\psi\colon G\times G\rightarrow \langle -1\rangle\cong {\bf Z}_2$ which satisfy
%\begin{equation}\label{eqcocyclic}
$$\psi(g_i,g_j)\psi(g_ig_j,g_k)=\psi(g_j,g_k)\psi(g_i,g_jg_k), \quad\forall g_i,g_j,g_k\in G$$
%\end{equation}
are
called {\it (binary) cocycles (over $G$)}\cite{McL95}.

A cocycle $\psi$ is naturally
displayed as {\it a cocyclic matrix} $M_\psi$;
 that is, the entry in the $(i,j)$th position of the cocyclic matrix is $\psi(g_i,g_j)$, for all $1\leq i,j\leq 4t$. i.e.,
 $$M_\psi=\left[\psi(g_i,g_j)\right]_{g_i,g_j\in G},$$
 where rows and columns of $M_\psi$ are indexed by the elements of $G$.
\begin{example}\label{cocyclesh}
Let $G={\bf Z}_2^k$. The vector inner product $\langle g_i,g_j\rangle$ determines a cocycle $\psi$, where $\psi(g_i,g_j)=(-1)^{\langle g_i,g_j\rangle}$ for all $g_i,g_j \in G$, and $M_\psi$ is the Sylvester Hadamard matrix of order $2^k$.
\end{example}

The  main purpose of this paper is to rewrite Ryser's formula to evaluate the permanent of Sylvester Hadamard matrices using their cocyclic properties which may lead to improvements. In \cite{Arm10}, an analogous approach has been done for computing the profile of cocyclic Hadamard matrices. As an example, we will focus on $H_8$.

\bigskip

\noindent{\bf Notation.} Throughout this paper we use $-$ for $-1$ and $1$ for $+1$. We write $H_{2^p}$ for a Sylvester Hadamard matrix of order $2^p$. The  cardinality of a set $S$ is denoted $\sharp S$. The notation $(0,1)$-matrix  means a matrix whose entries are either $0$ or  $1$. We use $I_n$ for the identity matrix of order $n$ and $M^T$ for the transpose of $M$.

\section{Ryser's formula for  Sylvester Hadamard matrices}

H.J. Ryser found an alternative method to evaluate the permanent of a matrix $A$ of order $n$, which is one of the fastest known general algorithms for computing permanents. By counting multiplications it has an efficiency of $O(2^nn)$. See \cite{Wan07} for some of the theories of permanents. See \cite{Hor07} for some elementary notions from cocyclic Hadamard matrices that we use here.

 If $A=[a_{ij}]$ is any $n\times n$ matrix
\begin{equation}\label{ryser'sf}
\mbox{per}\,(A)=\sum_{r=1}^n(-1)^{r}\sum_{\alpha\in Q_{r,n}}\prod_{j=1}^n \sum_{i\in \alpha} a_{i,j},
\end{equation}
where $Q_{r,n}$ denotes the set of all strictly increasing sequences of $r$ integers chosen from the set $\{1,2,\ldots, n\}$.

Here are three ways to calculate the permanent of  $$H_4=\left[\begin{array}{cccc}
1 & 1 & 1 & 1\\
1 & - & 1 & -\\
1 & 1 & - & -\\
1 & - & - & 1
\end{array}
\right].$$

%\begin{example}
The Classical formula, using all the permutations $S_4$, needs $23$ additions and $72$ multiplications.
Ryser's method needs  51 multiplications and 101 additions. Finally,
Algorithm \ref{algorithm2}, described below, needs only 9 multiplications and 4 additions.
%\end{example}
The last two methods are faster than the first one for larger matrices.

\begin{definition}\label{pequivalence}
Let $M_1$ and $M_2$ be two $(0,1)$-matrices. We say that $M_1$ and $M_2$ are {\it $P$-equivalent} if one is obtained from the other by a sequence of the operations:
 \begin{itemize}
 \item permutations of rows and/or columns.
  \item ``complementation'' of columns.
  \end{itemize}
  We mean by complementation of a column that every one of its entries change from $0$ to $1$ and from $1$ to $0$.
\end{definition}

\begin{definition}
We say that a $(0,1)$-matrix  $M$ is in {\em Strictly  increasing order by row} (or briefly SIOR)  if its $i$-th row is less than its $j$-th row (as binary numbers) when $i<j$.
\end{definition}

%\vspace{2mm}

\begin{algorithm} Searching for a distribution of $P$-equivalent classes.

\label{algorithm1}

\vspace{1.5mm}

\noindent{Input: Two positive integers $r$ and $p$.}

%\newline

\noindent{Output: The $P$-equivalent classes and the size of every class.}

\vspace{2.5mm}

\noindent{$\Omega_{r,p}\leftarrow \emptyset$}

%\newline

%\noindent{$\cal{S}\leftarrow$ all SIOR matrices $A\in \{0,1\}^{r\times p}$}

%\newline

\noindent{For each SIOR matrix $X$ }
  \hspace*{1cm}\begin{enumerate}

   \item[1.]  Check if  $X$  is $P$-equivalent
         to one element of $\Omega_{r,p}$.
             If no, go to 3; otherwise go to 2.

   \item[2.] $\sharp[Y]\leftarrow\sharp[Y]+1$ where $Y$ and $X$ are $P$-equivalent.
      End.

     \item[3.]  $\,\Omega_{r,p}\leftarrow\Omega_{r,p}\cup [X]$ and $\sharp[X]\leftarrow 1$. End.
                  \end{enumerate}

\vspace{2mm}

\noindent{$\Omega_{r,p}=\{[X_1],\ldots,[X_k]\}$ and $\Omega_{r,p}^\sharp=\{\sharp[X_1],  \sharp[X_2], \ldots,  \sharp[X_k]\}$. }
\end{algorithm}

\bigskip

\noindent{\sc Verification:}
By construction, the uniqueness of $[Y]$ in the Step 2 is guaranteed, $\Omega_{r,p}$ is a set of equivalence classes and  $\Omega_{r,p}^\sharp$ gives the size of every orbit.

\bigskip

The following result can be seen as an immediate consequence of the procedure described above.

\begin{proposition}\label{proposicionsuma} Assuming that $\Omega_{r,p}=\{[X_1],\ldots,[X_k]\}$ is a distribution of equivalence classes for the pair $(r,p)$. We have:
 $$\bullet\,\,\displaystyle\sum_{i=1}^n  \sharp[X_i]=\left(\begin{array}{c} 2^p \\ r\end{array}\right),\qquad
\bullet\,\,\sharp[X_i]\leq 2^p p!,\qquad
\bullet\,\,\sharp\Omega_{r,p}\geq \frac{\left(\begin{array}{c} 2^p \\ r\end{array}\right)}{2^p p!}. $$
\end{proposition}
\begin{proof}
The first identity follows from the fact that the number of  SIOR matrices $A\in \{0,1\}^{r\times p}$ is $\left(\begin{array}{c} 2^p \\ r\end{array}\right).$
On the other hand, attending to Definition \ref{pequivalence}, it is clear that the size of the $P$-orbit for any $X_i$ is bounded by $2^p p!$.
Finally, the last inequality is a straightforward consequence of the two previous results in this proposition.
\end{proof}

Now we study the connection between the permanent of Sylvester matrices and  $P$-equivalence. Before that, we need some definitions and notations.
 From now on, we write
 $g_{i}$ for the element of ${\bf Z}_2^p$ which is the binary representation of the integer $i-1,$ where $i\in N$. Then
 $M_\alpha$ denotes the SIOR $r\times p$ matrix with entries $\{0,1\}$ where its $l$-th row is $g_{i_l}$. That is,
    $$M_\alpha=\left(\begin{array}{c}
    g_{i_1}\\
    g_{i_2}\\
    \vdots \\
    g_{i_r}
    \end{array}\right)$$ where $\alpha=\{ i_1<i_2<\ldots<i_r\}\in Q_{r,n}$.
\begin{remark}\label{remark1to1}
There is a one to one correspondence between SIOR $r\times p$ matrices $M$ with entries $\{0,1\}$ and subsets $\alpha=\{ i_1<i_2<\ldots<i_r\}\in Q_{r,n}$ where $n=2^p$. Sometimes, we denote $\alpha$ by $\alpha_{\scst M}$.
Let us observe that the $i_k$-row/column of $H_{2^p}$ is indexed by $g_{i_k}$ as a cocyclic matrix (see Example \ref{cocyclesh}).
\end{remark}

\begin{theorem}\label{invariant}
Assuming that $H=[h_{i,j}]$ is a Sylvester Hadamard matrix of order $n=2^p$ and $M \in \{0,1\}^{r\times p}$ is a  SIOR matrix. Let us define
$$\Phi(\alpha_{\scst M})= \prod_{j=2}^n \sum_{i\in {\alpha_{\scst M}}} h_{i,j}.$$
Then,
\begin{enumerate}
   \item $\Phi(\alpha_{\scst M})$ is invariant under rows permutation of $M$.
   \item $\Phi(\alpha_{\scst M})$ is invariant under columns permutation of $M$.
   \item  $\Phi(\alpha_{\scst M})$ is invariant under complementations of columns of $M$.
\end{enumerate}

\end{theorem}
\begin{proof}
Let us observe that,
$$h_{i,j}=\psi(g_i,g_j), \,\,1\leq i,j\leq n\quad \mbox{(where $\psi$ was defined in Example \ref{cocyclesh})}.$$
Given $P$ and $Q$ permutation matrices of order $n$. Then:
\begin{enumerate}
\item
Obviously, any permutation in the rows of $M$ corresponds with the same subset of $r$ distinct rows of  $H_{2^p}$, i.e., $\alpha_{\scst M}=\alpha_{\scst PM}$. Hence, $\Phi(\alpha_{\scst M}) =\Phi(\alpha_{\scst PM})$

\item  A permutation of columns of $M$ may correspond with a different subset of $r$ distinct rows of $H_{2^p}$, i.e., $\alpha_{\scst M}\neq \alpha_{\scst MQ}$ in general. Since $$\psi (g_i,g_j)=\psi(g_iQ,g_jQ),$$ thus  $\Phi(\alpha_{\scst M}) =\Phi(\alpha_{\scst MQ})$.

\item
 Again, a complementation of the $k$-th column of $M$ may correspond with a different subset of $r$ distinct rows of $H_{2^p}$, i.e., $\alpha_{\scst M}\neq \alpha_{\scst MI(k)}$ in general. But
 $$\psi (g_iI(k),g_j)=\left\{\begin{array}{rc}
 \psi( g_i,g_j) & \quad \mbox{if $k$-th entry of $g_j$ is 0} \\
 -\psi(g_i,g_j) & \quad \mbox{if $k$-th entry of $g_j$ is 1,}\end{array}\right.$$
 $I(k)$ denotes a diagonal matrix obtained from
the identity matrix by negating the k-th diagonal entry
 and $g_iI(k)$ means the element of ${\bf Z}_2^p$ where  the $k$-th entry of $g_i$ changes from 0 to 1 or viceversa and the other entries remain the same.
 From the fact that the number of elements of ${\bf Z}_2^p$ with 1 in the $k$-th entry is $2^{p-1}$, it follows that
  $\Phi(\alpha_{\scst M}) =\Phi(\alpha_{\scst MI(k)})$.

\end{enumerate}
\end{proof}
\begin{corollary}\label{corolario1}
Let $M_1, M_2 \in \{0,1\}^{r\times p}$ be  SIOR matrices and $H=[h_{i,j}]$ be the Sylvester Hadamard matrix of order $n=2^p$. If $M_1$ is $P$-equivalent to $M_2$ then
    $$ \Phi(\alpha_{\scst M_1})=\Phi(\alpha_{\scst M_2}).$$
\end{corollary}

\begin{proposition}
Let $H=[h_{i,j}]$ be the Sylvester Hadamard matrix of order $n=2^p$. Ryser's formula for $H$ can be rewritten as
\begin{equation}\label{rformulacociclica}
 \sum_{r=1}^n (-1)^rr\sum_{i=1}^{\sharp\Omega_{r,p}}\sharp[X_i^r]\Phi(\alpha_{\scst X_i^r})
 \end{equation}
where $\Omega_{r,p}=\{[X_1^r],\ldots,[X^r_{k_r}]\}$ is the output of Algorithm \ref{algorithm1}.
\end{proposition}
\begin{proof}
By definition,  Ryser's formula for $H$ is
$$\sum_{r=1}^n (-1)^r r \sum_{\alpha\in Q_{r,n}}\Phi(S).$$  Taking into account Corollary \ref{corolario1}, it takes a simple inspection to see that
$$\sum_{\alpha\in Q_{r,n}}\Phi(S)= \sum_{i=1}^{\sharp\Omega_{r,p}}\sharp[X_i^r]\Phi(\alpha_{\scst X_i^r}), \quad \forall r=1,\ldots,n.$$
\end{proof}

\begin{remark}Let $H$ be a normalized Hadamard matrix of order $n$. Then every column/row sum vanishes but the first which is $n$.
\end{remark}
As a consequence,
$$
\begin{array}{lc}
\displaystyle\sum_{i\in \alpha} h_{i,1}= r, &  \\[5mm]
\displaystyle\sum_{i\in {\alpha}} h_{i,j}=-\sum_{i\in \bar{\alpha}} h_{i,j}, & j\geq 2;
\end{array}$$
for all $\alpha\in Q_{r,n}$ where $\bar{\alpha}\in Q_{n-r,n}$ and $\alpha\cup \bar{\alpha}=\{1,2,\ldots,n\}$.
Taking into account the previous identities, the formula (\ref{rformulacociclica}) can be rewritten as
\begin{equation}\label{rformulacociclicas}
\sum_{r=1}^{\frac{n}{2}} (-1)^r(2r-n)\sum_{i=1}^{\sharp\Omega_{r,p}}\sharp[X_i^r]\Phi(\alpha_{\scst X_i^r}).
\end{equation}

\vspace{2mm}

\begin{algorithm}\label{algorithm2} Evaluating the permanent of $H_{2^p}$ by means of Ryser's formula

\vspace{2mm}

\noindent{Input: $H$ be a Sylvester Hadamard matrix of order $n=2^p$.}

%\newline

\noindent{Output: $per (H)$.
}

\vspace{2mm}
\begin{enumerate}
\item[] Step 1.  Preprocess:  Use Algorithm \ref{algorithm1} to compute $\Omega_{r,p}$ and $\Omega^\sharp_{r,p}$
for $\mbox{$ r= 1,\ldots,2^{p-1}$}$.
\item[] Step 2. Evaluate (\ref{rformulacociclicas}).
\end{enumerate}
\end{algorithm}

\section{Example: $H_8$}
In this section we deal with the problem of computing the permanent of Sylvester Hadamard matrix of order $8$,

$$H_8=\left[\begin{array}{cccccccc}
1 & 1 & 1 & 1 & 1 & 1 & 1 & 1 \\
1 & - & 1 & - & 1 & - & 1 & - \\
1 & 1 & - & - & 1 & 1 & - & - \\
1 & - & - & 1 & 1 & - & - & 1 \\
1 & 1 & 1 & 1 & - & - & - & - \\
1 & - & 1 & - & - & 1 & - & 1 \\
1 & 1 & - & - & - & - & 1 & 1 \\
1 & - & - & 1 & - & 1 & 1 & -
\end{array}\right].$$
The aim of this section is to show the difference in the number of elementary operations  required in Ryser's formula (\ref{ryser'sf}) and in Algorithm \ref{algorithm2} for computing $\mbox{per}\,(H_8)$.

First, a useful lemma.
\begin{lemma}\label{removing}
Let $H$ be a Sylvester Hadamard matrix of order $n=2^p$ (p=1,2,3,4). Let $H_{k,p}$ be a $2^k\times 2^p$ matrix  obtained from $H$ by removing  $2^p-2^k$ of its rows  (where $k=1,\ldots,{p-1}$). Then
for at least one column of $H_{k,p}$, its column sum vanishes.
\end{lemma}
\begin{proof}
Until $p=3$ the proof is by induction, for $p=4$ it is seen by inspection.
\end{proof}

\begin{corollary} \label{pares}
If $p=1, 2, 3$ or $4$  and $r=2^k$ with $k=1,\ldots,{p-1}$ then $\Phi(\alpha_{\scst X_i^r})=0$ for all $[X_i^r]\in \Omega_{r,p}$.
\end{corollary}

Now, we are going to use Algorithm \ref{algorithm2} to compute $\mbox{per}\,(H_8)$.

\begin{enumerate}
\item[] Step 1.  Preprocess:  Use Algorithm \ref{algorithm1} to compute $\Omega_{r,p}$ and $\Omega^\sharp_{r,p}$.
Due to Corollary \ref{pares} is only necessary for $r=1,3$. See Table 1 for the Output.

\begin{center}\begin{table}$$\begin{array}{|c||c|c||} \hline r & \mbox{inequivalent matrices} &\# \mbox{orbits} \\ \hline 1& (0,0,0), \,(0,0,1)  &2, 6\\
\hline 3 &
\left[\begin{array}{ccc} 0 & 0 & 0 \\
0 & 0 & 1 \\
0 & 1 & 0 \\
\end{array}\right],\,
\left[\begin{array}{ccc} 0 & 0 & 0 \\
0 & 0 & 1 \\
1 & 1 & 0 \\
\end{array}\right],\,
\left[\begin{array}{ccc} 0 & 0 & 0 \\
0 & 1 & 1 \\
1 & 0 & 1 \\
\end{array}\right]

& 24,24,8  \\ \hline \end{array}$$\caption{Output of Algorithm 1 for $p=3$ and $r=1,3$.}\end{table}\end{center}

 \item[] Step 2. Evaluate formula \ref{rformulacociclicas}.
$$\mbox{per}\,(H_8)=\,\,(-1)^{1}(2-8)
\sum_{i=1}^{\sharp\Omega_{1,3}}\sharp[X_i^1]\Phi(\alpha_{\scst X_i^1})+(-1)^{3}(6-8)
\sum_{i=1}^{\sharp\Omega_{3,3}}\sharp[X_i^3]\Phi(\alpha_{\scst X_i^3})\,\,$$
(Using Table 1, we get)
$$=\,\,6\cdot(2\cdot 1+6\cdot 1)+2\cdot(24\cdot3 + 24\cdot3 + 8\cdot3 )=384.\qquad\qquad\quad$$

\end{enumerate}
Finally, Table 2  compares the number of elementary operations (additions and multiplications) that it is  needed to compute the permanent of $H_8$ by Ryser's formula (see (\ref{ryser'sf})) and Algorithm \ref{algorithm2}.

\begin{center}\begin{table}$$\begin{array}{|c||c|c||} \hline  & \mbox{Ryser's formula} & \mbox{Algorithm 2} \\ \hline
\# Oper. & 9913  & \begin{array}{r}
 \mbox{Step 1.}\,\,\, 2688\\
 \mbox{Step 2.}\,\,\,\quad  91\\[2mm]
                 2779         \end{array}\\
\hline \end{array}$$\caption{$\#$ of operations that it is needed to compute $\mbox{per}\,(H_8$).}\end{table}\end{center}

\section[]{Conclusions and further work}
In this article, we have rewritten  Ryser's formula for computing the permanent of Sylvester Hadamard matrices attending to their cocyclic properties. The rewritten  formula presents  an important reduction in the number of sets of $r$ distinct rows of $H_{2^p}$ which are involved. However, from a more practical point of view,  Algorithm \ref{algorithm2} is only appropriate for numerical calculations in low orders
 because the method relies on Algorithm \ref{algorithm1}, and this is generally time-consuming.  For low orders, Algorithm \ref{algorithm2} seems to be faster than  Ryser's. Whereas in greater orders, both algorithms  present similar limitations. Hence, analytical formulas for the permanents of Sylvester Hadamard are still far away.

Some improvements of our method could be:
\begin{enumerate}
\item Study the spectrum of the function $\Phi(\alpha)$ and search for a more efficient characterization of the set of all $\alpha\in Q_{r,n}$ with the same value for $\Phi$. This could speed up our Algorithm.

\item Study if for each $p>4$  Lemma \ref{removing}  is true.
\end{enumerate}

Let us remark that for $H_{1,p}$ and $H_{2,p}$ the thesis of Lemma \ref{removing} holds true $\forall\,p\geq 3$.
 It is a straightforward consequence of the orthogonality of $H_{2^p}$. Moreover,
an affirmative answer to the second improvement  implies the following result.

\begin{proposition}
Let $G$ be a $p\times 2^{p-1}$ matrix where its columns are different binary vectors of length $p$. If $G$ is a generator matrix of a binary $(2^{p-1},p)$ code $C$, then there exists at least one codeword of $C$ with weight equal to $2^{p-2}$.
\end{proposition}

\begin{proof}
Let $H_{p-1,p}$ be the $2^{p-1}\times 2^{p}$ matrix   obtained from $H_{2^p}=\left[\psi(g_i,g_j)\right]_{g_i,g_j\in{\bf Z}_2^p}$ by removing  $2^{p-1}$ of its rows. Concretely, those rows index by the elements of ${\bf Z}_2^p$ different from the columns of $G$.
Let $g_i$ be the element of ${\bf Z}_2^p$ indexing the column of $H_{p-1,p}$ such that  its column sum vanishes. Then,
$g_iG$ is a codeword of $C$ with  weight equal to $2^{p-2}$. Since swapping the  $0$ entries of $(g_iG)^T$  to  $-1$, it becomes  the $i$-th column of $H_{p-1,p}$.
\end{proof}

 \subsection{Acknowledgements}
  I thank V\'ictor \'Alvarez for numerous interesting discussions and for his assistance with the computational aspects
of the paper. I thank  Kristeen Cheng for her reading of the manuscript.

This work has been partially supported by the research projects FQM-016 and P07-FQM-02980 from JJAA and MTM2008-06578 from MIC\-INN (Spain) and FEDER (European Union).

%%%%%%%%%%%%%%%%%%%%%%%%%%%%%%%%%%%%%%%%%%%%%%%%%%%%%%%%%%%%%

\end{document}